\documentclass[11pt]{amsart}

\usepackage[T1]{fontenc}
\usepackage[latin1]{inputenc}
\usepackage[english]{babel}
\usepackage{amsmath,amssymb,amsthm}
\usepackage{enumitem}
\renewcommand{\Re}{\operatorname{Re}}
\renewcommand{\Im}{\operatorname{Im}}
\newcommand{\cH}{\mathcal H}

\usepackage{hyperref}

\renewcommand{\MR}[1]{}
\newcommand{\mcc}{M\textsuperscript{c}Carthy}
\newcommand{\Mult}{\operatorname{Mult}}
\newcommand{\bD}{\mathbb D}
\newcommand{\bB}{\mathbb B}

\newtheorem{theorem}{Theorem}[section]
\newtheorem{lemma}[theorem]{Lemma}
\newtheorem{proposition}[theorem]{Proposition}
\newtheorem{corollary}[theorem]{Corollary}

\theoremstyle{definition}

\newtheorem{question}[theorem]{Question}
\newtheorem{remark}[theorem]{Remark}
\newtheorem*{remark*}{Remark}

\title[Embedding dimension]{Embedding dimension of the Dirichlet space}

\author[M. Hartz]{Michael Hartz}
\address{Fachrichtung Mathematik, Universit\"at des Saarlandes, 66123 Saarbr\"ucken, Germany}
\email{hartz@math.uni-sb.de}
\thanks{The author was partially supported by a GIF grant.}


\date{\today}

\begin{document}

\begin{abstract}
  The classical Dirichlet space is a complete Pick space, hence by a theorem of Agler and \mcc, there exists
  an embedding $b$ of the unit disc into a $d$-dimensional ball such that composition
  with $b$ realizes the Dirichlet space as a quotient of the Drury--Arveson space.
  We show that $d =\infty$ is necessary, even if we only demand that composition with $b$
  induces a surjective map between the multiplier algebras.
\end{abstract}

\maketitle

\section{Introduction}

\subsection{Background}
The classical Dirichlet space on the open unit disc $\mathbb{D} \subset \mathbb{C}$ is defined
as
\begin{equation*}
  \mathcal{D} = \Big\{f \in \mathcal{O}(\mathbb{D}) : \int_{\mathbb{D}} |f'|^2 d A < \infty \Big\},
\end{equation*}
where $d A$ denotes the normalized area measure on $\mathbb{D}$.
Equipped with the norm
\begin{equation*}
  \|f\|^2 = \int_{\mathbb{D}} |f'|^2 d A + \|f\|_{H^2}^2,
\end{equation*}
where $H^2$ denotes the classical Hardy space and
\begin{equation*}
  \|f\|^2_{H^2} = \sup_{0 \le r < 1} \int_{0}^{2 \pi} |f(r e^{ it})|^2 \frac{dt}{2 \pi},
\end{equation*}
the Dirichlet space is a reproducing kernel Hilbert space of functions on $\mathbb{D}$ with reproducing kernel
\begin{equation*}
  k(z,w) = \frac{1}{z \overline{w}} \log \Big( \frac{1}{1 - z \overline{w}} \Big).
\end{equation*}
For background and motivation, the reader is referred to the books \cite{EKM+14,ARS+19}.

An important feature of the Dirichlet space is the complete Pick property, meaning that $\mathcal{D}$ satisfies
a version of the classical Pick interpolation theorem; see \cite{AM02} for background.
This key realization, due to Agler, has led to the solution of some open problems in the Dirichlet space,
such as the characterization of interpolating sequences due to Marshall and Sundberg \cite{MS94a}
(a different proof was given independently by Bishop \cite{Bishop94}), quotient representations of functions in $\mathcal{D}$ by Aleman, \mcc, Richter and the author \cite{AHM+17a}, and factorization in the weak product space
by Jury and Martin \cite{JM18}.

A theorem of Agler and \mcc\ \cite{AM00} shows that the complete Pick property of $\mathcal{D}$ is equivalent
to the existence of a number $d \in \mathbb{N} \cup \{\infty\}$ and a map
$b: \mathbb{D} \to \mathbb{B}_d$, so that
\begin{equation}
  \label{eqn:AM_kernel}
  k(z,w) = \frac{1}{1 - \langle b(z),b(w) \rangle } \quad (z,w \in \mathbb{D}).
\end{equation}
Here, $\mathbb{B}_d$ is the open unit ball in $\mathbb{C}^d$, which is understood as $\ell^2$ if $d= \infty$.
(In general, one needs to allow rescalings of the form $\frac{\delta(z) \overline{\delta(w)}}{1 - \langle b(z), b(w) \rangle}$
for a nowhere vanishing function $\delta$, but one can achieve that $\delta = 1$ since $k(z,0) = 1$ 
for all $z$.)
In fact, the complete Pick property of $\mathcal{D}$ is often applied by using that \eqref{eqn:AM_kernel} holds,
as in the works \cite{AHM+17a,JM18} mentioned above.
We will explicitly construct such a map $b$ in Section \ref{sec:explicit_embedding}.

The relation \eqref{eqn:AM_kernel} can be interpreted in the following way.
Let $K(z,w) = \frac{1}{1 - \langle z,w \rangle }$ for $z,w \in \mathbb{B}_d$. This is a positive definite
kernel, and the corresponding reproducing kernel Hilbert space is denoted by $H^2_d$ and called the Drury--Arveson space \cite{Arveson98}.
The theorem of Agler and \mcc\ can be interpreted as saying that $H^2_d$ is a universal complete Pick space.
In particular, in the case of the Dirichlet space, \eqref{eqn:AM_kernel} can be rephrased by saying that
$b$ induces a co-isometric composition operator
\begin{equation*}
  H^2_d \to \mathcal{D}, \quad f \mapsto f \circ b;
\end{equation*}
see Proposition \ref{prop:other_emb} for more details.
As a consequence, on the level of multiplier algebras, $b$ induces
a (complete) quotient mapping
\begin{equation*}
  \Phi: \Mult(H^2_d) \to \Mult(\mathcal{D}), \quad \varphi \mapsto \varphi \circ b,
\end{equation*}
meaning that the induced map $\Mult(H^2_d) / \ker(\Phi) \to \Mult(\mathcal{D})$ is a (completely) isometric
isomorphism. In particular, $\Phi$ is surjective.
These implications will also be reviewed in Proposition \ref{prop:other_emb}.

It turns out that $H^2_d$ and its multiplier algebra are significantly more tractable if $d < \infty$.
To give a concrete example, the maximal ideal space $\mathcal{M}(\Mult(H^2_d))$ of $\Mult(H^2_d)$ is fibered over $\overline{\mathbb{B}_d}$ via the map
\begin{equation*}
  \pi: \mathcal{M}(\Mult(H^2_d)) \to \overline{\mathbb{B}_d}, \quad \chi \mapsto (\chi(z_1),\chi(z_2),\ldots).
\end{equation*}
If $d< \infty$, the fibers over the open ball $\mathbb{B}_d$ are singletons.
But if $d = \infty$, the fibers are extremely complicated. Indeed, $\pi^{-1}(0)$ contains
a copy of the Stone-\v{C}ech remainder $\beta \mathbb{N} \setminus \mathbb{N}$ if $d = \infty$.
This has led to some issues in the area; cf.\ \cite{DHS15a}.

\subsection{Main result}
Since $H^2_d$ is more tractable if $d < \infty$,
the question arose for which complete Pick spaces one can achieve $d < \infty$ in the Agler--\mcc\ theorem.
In particular, it was asked whether this is possible for the Dirichlet space, see \cite[Section 7.5]{SS14}.
It was shown independently by Rochberg \cite{Rochberg16} and by the author \cite[Corollary 11.9]{Hartz17a} that
an embedding $b: \mathbb{D} \to \mathbb{B}_d$ for which relation \eqref{eqn:AM_kernel} holds
only exists if $d = \infty$.
However, for a number of questions, a weaker relation than \eqref{eqn:AM_kernel} would be sufficient.
For instance, one might only demand that
$b$ induces a surjective composition operator $H^2_d \to \mathcal{D}$.
Perhaps one of the weakest reasonable notions of embedding is to demand that
$b$ induces a surjective map $\Mult(H^2_d) \to \Mult(\mathcal{D})$,
which is still sufficient
for understanding algebraic properties of $\Mult(\mathcal{D})$, such as the maximal ideal space.
The relationship between this and other notions of embedding will be explained in Proposition \ref{prop:other_emb}.

The main result of this note shows that even with this weak notion of embedding, $d = \infty$ is necessary.
In particular, this answers the question raised by Salomon and Shalit in \cite[Section 7.5]{SS14}.
For a statement adapted to the setting of \cite{SS14}, see also Theorem \ref{thm:no_iso}.

\begin{theorem}
  \label{T:main}
  There does not exist a map $b: \mathbb{D} \to \mathbb{B}_d$ with $d \in \mathbb{N}$
  that induces a surjective homomorphism
  \begin{equation*}
    \Mult(H^2_d) \to \Mult(\mathcal{D}), \quad \varphi \mapsto \varphi \circ b.
  \end{equation*}
\end{theorem}

This result will be proved in Section \ref{sec:proof}.
The arguments involving ranks of kernels, which were used in \cite{Hartz17a} to show that
\eqref{eqn:AM_kernel} cannot hold for $d < \infty$, do not seem to generalize to the setting of Theorem \ref{T:main}.
Instead, we will refine Rochberg's geometric arguments of \cite{Rochberg16}.
He observed that if \eqref{eqn:AM_kernel} holds, then $b$ is an isometric embedding
from the disc, equipped with a suitable metric related to the Dirichlet space, into the ball,
equipped with the pseudohyperbolic metric. He then used geometric arguments to show that this is impossible for $d < \infty$.

The proof of Theorem \ref{T:main} follows a similar outline. First, we will show that if $b: \mathbb{D} \to \mathbb{B}_d$
induces a surjective homomorphism as in Theorem \ref{T:main}, then $b$ has to be a bi-Lipschitz embedding
with respect to the metrics mentioned above.
We will then identify geometric obstructions showing that this is impossible if $d < \infty$.
The basic idea of the geometric argument is similar to that in Rochberg's proof,
but the bi-Lipschitz setting necessitates significant changes compared to the isometric setting.

\subsection{An explicit embedding into \texorpdfstring{$\mathbb{B}_\infty$}{Binfinity}}
In Section \ref{sec:explicit_embedding}, we will explicitly construct a map $b: \mathbb{D} \to \mathbb{B}_\infty$
such that the reproducing kernel $k$ of the Dirichlet space has the representation
\begin{equation*}
  k(z,w) = \frac{1}{1 - \langle b(z), b(w) \rangle} \quad (z,w \in \mathbb{D}).
\end{equation*}
As is well known, the key point in proving the existence of an embedding is to show that the real numbers
$(c_n)_{n=1}^\infty$ defined by the power series identity
\begin{equation}
  \label{eqn:c_n}
  \sum_{n=1}^\infty c_n (z \overline{w})^n = 1 - \frac{1}{k(z,w)}
\end{equation}
satisfy $c_n \ge 0$ for all $n \ge 1$.
In this case, considering \eqref{eqn:c_n} on the diagonal $z=w$, we find that $\sum_{n=1}^\infty c_n \le 1$,
so defining
\begin{equation*}
  b: \mathbb{D} \to \mathbb{B}_\infty, \quad b(z) = ( \sqrt{c_1} z, \sqrt{c_2} z^2, \sqrt{c_3} z^3 ,\ldots),
\end{equation*}
we obtain the desired embedding.

In the literature on complete Pick spaces, the usual proof of the fact that the numbers $c_n$ in \eqref{eqn:c_n} are non-negative makes
use of a lemma of Kaluza; see \cite{MS94a} or \cite[Corollary 7.41]{AM02}.
While the coefficients $c_n$ can be obtained from the Taylor coefficients of $k$
through a recursive formula, this does not give an explicit formula for $c_n$.
In the discussion following Corollary 11 in \cite{MS94a}, Marshall and Sundberg
ask if it is possible to see directly from from the formula for the kernel $k$ that $c_n \ge 0$
for all $n \ge 1$.

We show that an argument from the work of Kluyver \cite{Kluyver24} yields an explicit formula for $c_n$, which shows in particular that $c_n \ge 0$
for all $n \ge 1$.

\begin{proposition}
  \label{prop:explicit_intro}
  For $n \ge 1$, we have
  \begin{equation*}
    c_n = \int_0^1 t \frac{\Gamma(n-t)}{\Gamma(n+1) \Gamma(1-t)} \, dt \ge 0.
  \end{equation*}
\end{proposition}
This result will be proved in Section \ref{sec:explicit_embedding}.

\subsection{Reader's guide}
The remainder of this note is organized as follows.
In Section \ref{sec:prelim}, we recall some necessary preliminaries for the proof
of the main result, which occupies Section \ref{sec:proof}.
In Section \ref{sec:other}, we show that Theorem \ref{T:main} also rules out the existence
of some other types of embedding for the Dirichlet space. Finally, in Section \ref{sec:explicit_embedding},
we prove Proposition \ref{prop:explicit_intro}.

\begin{remark*}
A draft version of this note was circulated in 2016.
Since the draft has been cited a few times, the author decided to make this note more widely available.
\end{remark*}

\section{Preliminaries about metrics}
\label{sec:prelim}

In this section, we collect the necessary background on metrics induced by reproducing kernels.
More information on this topic can be found in \cite{ARS+11} and \cite[Chapter 9]{AM02}.

\subsection{Metrics induced by kernels}
Let $\mathcal{H}$ be a reproducing kernel Hilbert space of functions on a set $X$
with reproducing kernel $K$.
Throughout, we will assume that $K(z,z) \neq 0$ for all $z \in X$.
We say that $\mathcal{H}$ \emph{separates the points of $X$}
if whenever $z,w \in X$ with $z \neq w$, there exists $f \in \mathcal{H}$ with $f(z) = 0$ and $f(w) =1$.
This is equivalent to saying that $K(\cdot,z)$ and $K(\cdot,w)$ are linearly independent whenever
$z \neq w$.

We can define a pseudo-metric $\delta_{\mathcal{H}}$ on $X$ by
\begin{equation}
  \label{eqn:metric}
  \delta_{\cH}(z,w) = \Big({1-\frac{|K(z,w)|^2}{K(z,z)K(w,w)}} \Big)^{1/2}.
\end{equation}
This is a metric if $\mathcal{H}$ separates the points of $X$, see \cite[Lemma 9.9]{AM02}
or \cite[Section 4]{ARS+11}.

If $\mathcal{H} = H^2$, the Hardy space on the disc, then a computation with Szeg\H{o} kernels
shows that
\begin{equation*}
  \delta_{H^2}(z,w) = \Big| \frac{z - w}{ 1 - \overline{z} w} \Big| \quad (z,w \in \mathbb{D}),
\end{equation*}
so $\delta_{H^2}$ is the pseudohyperbolic metric on $\mathbb{D}$; see \cite[Section 2]{ARS+11}.

More generally, the metric $\delta_{H^2_d}$ associated with the Drury--Arveson
space turns out to be the pseudohyperbolic metric on $\mathbb{B}_d$ for $d \in \mathbb{N}$.
Background information on the pseudohyperbolic metric on the ball can be found in \cite{DW07}.
To recall, we have for each $a \in \mathbb{B}_d$ the biholomorphic
automorphism
\begin{equation*}
  \varphi_a: \mathbb{B}_d \to \mathbb{B}_d, \quad \varphi_a(z) = \frac{a - P_a(z) - s_a Q_a(z)}{1 - \langle z,a \rangle },
\end{equation*}
where $P_a$ is the orthogonal projection onto $\mathbb{C} a$, $Q_a = I - P_a$ and $s_a = (1 - \|a\|^2)^{1/2}$;
see \cite[Section 2.2]{Rudin08}.
The identity
\begin{equation*}
  1 - \|\varphi_w(z)\|^2 = \frac{(1 -\|w\|^2)(1-\|z\|^2)}{| 1 - \langle z,w \rangle|^2 }
\end{equation*}
(see \cite[Theorem 2.2.2 (iv)]{Rudin08}) shows that
\begin{equation*}
  \delta_{H^2_d}(z,w) = \|\varphi_w(z)\|,
\end{equation*}
which by definition is the pseudohyperbolic distance between $z$ and $w$,
usually denoted by $\rho(z,w)$.

In the case of the Dirichlet space, we obtain the formula
\begin{equation}
  \label{eqn:Dirichlet_metric}
  \delta_{\mathcal{D}}(z,w) = \Big( 1 - \frac{|\log(1  - \overline{z}w)|^2}{\log(1 - |z|^2) \log(1 - |w|^2)} \Big)^{1/2}
\end{equation}
for $z,w \in \mathbb{D} \setminus \{0\}$.

\subsection{Length of curves}
\label{ss:length}

The metrics $\delta_{\mathcal{H}}$ defined in \eqref{eqn:metric} will not be quite sufficient
in our setting. In addition, we will make use of length metrics induced
by the metrics $\delta_{\mathcal{H}}$.
Background on this topic can be found in \cite{ARS+11} and \cite{MPS85}, see also \cite{JP13}.

If $(X,d)$ is a metric space and $\gamma: [a,b] \to X$ is a (continuous) curve,
the \emph{length}  of $\gamma$ with respect to $d$ is defined by
\begin{equation*}
  \ell_d(\gamma) = \sup
  \left\{ \sum_{j=0}^{n-1} d(\gamma(t_j),\gamma(t_{j+1})) : a = t_0 < t_1 < \ldots < t_n = b \right\}.
\end{equation*}
We also set
\begin{equation*}
  d^*(z,w) = \inf \{ \ell_d(\gamma): \gamma \text{ is a curve joining $z$ to $w$} \}.
\end{equation*}

It is well known that if $\rho$ is the pseudo-hyperbolic metric on $\mathbb{B}_d$, then $\rho^*$ is the Poincar\'e--Bergman metric
\begin{equation*}
  \beta(z,w) = \tanh^{-1} \rho(z,w) = \frac{1}{2} \log \Big( \frac{1 + \rho(z,w)}{1 - \rho(z,w)} \Big).
\end{equation*}
We will indicate below how to see this fact in the present framework of metrics induced by kernels;
see also \cite[Chaper 1]{JP13} for a direct proof in the case $d=1$.

Mazur, Pflug and Skwarzcy\'nski \cite{MPS85} showed that if the metric $d$ is derived from the Bergman kernel
of a domain in $\mathbb{C}^d$
in a manner similar to \eqref{eqn:metric}, then the associated metric $d^*$
is essentially the Bergman metric of the domain, which is a Riemannian
metric that can be explicitly computed from the Bergman kernel;
see for instance \cite[Section 1.4]{Krantz92} for background on the Bergman metric.
Arcozzi, Rochberg, Sawyer and Wick \cite{ARS+11} observed that the results of \cite{MPS85}
extend to much more general spaces of holomorphic functions.
We will make use of these results in the case of the Dirichlet space and of the Drury--Arveson space.

\begin{lemma}
  \label{lem:metrics}
  \leavevmode
  \begin{enumerate}[label=\normalfont{(\alph*)},wide]
    \item For the pseudohyperbolic metric $\rho$ on $\mathbb{B}_d$, the metric $\rho^*$ coincides with the Poincar\'e--Bergman metric $\beta$.
    \item Let $\gamma: [a,b] \to \mathbb{D}$ be a piecewise $C^1$ curve. Then the length of $\gamma$ with respect to the Dirichlet
      space metric $\delta_{\mathcal{D}}$ is given by
  \begin{equation*}
    \ell_{\delta_{\mathcal{D}}}(\gamma) = \int_{a}^b g(\gamma(t))^{1/2} |\gamma'(t)| \, d t,
  \end{equation*}
  where $g: \mathbb{D} \to [0,\infty)$ is the continuous function given by
  \begin{equation*}
    g(z) = 
    \frac{ \log \big( \frac{1}{1 - |z|^2} \big)- |z|^2}{ \big(\log \big( \frac{1}{1 - |z|^2} \big) \big)^2 (1-|z|^2)^2} \quad (z \neq 0)
  \end{equation*}
  and $g(0) = \frac{1}{2}$.
  \end{enumerate}
\end{lemma}

\begin{proof}
  It was shown in \cite[Theorem 1]{MPS85} that if $K$ is the square root of the kernel of the Bergman space on a bounded domain $D$ in $\mathbb{C}^d$
  and if $\delta_{\mathcal{H}}$ is defined as in \eqref{eqn:metric}, then the length of any piecewise $C^1$ curve $\gamma: [a,b] \to D$
  is given by
  \begin{equation}
    \label{eqn:MPS}
    \ell_{\delta_\mathcal{H}}(\gamma) = \int_{a}^b \Big( \sum_{j,k=1}^d  \frac{\partial^2 \log K(\gamma(t),\gamma(t))}{\partial \overline{z_j} \partial z_k}
    \overline{\gamma_j'(t)} \gamma_k'(t) \Big)^{1/2} \,d t.
  \end{equation}
  (The result in \cite{MPS85} involves additional factors of $\frac{1}{2}$ and $\sqrt{2}$ because they consider the square root of the reproducing kernel.)
  It was observed in \cite[Proposition 9]{ARS+11} that this formula, which is proved by second order Taylor approximations to the kernel functions,
  holds much more generally, and in particular in the case of the Drury--Arveson space and of the Dirichlet space.

  (a) Applying \eqref{eqn:MPS} to the Drury--Arveson kernel $K$ and recalling that the Bergman kernel on $\mathbb{B}_d$ is given by $K^{d+1}$,
  it follows from standard results about the Poincar\'e--Bergman metric (see, for example, \cite[Proposition 1.21]{Zhu05}) that
  \begin{equation*}
    \beta(z,w) = \inf \{ \ell_{\rho}(\gamma): \gamma \text{ is a piecewise smooth curve in $\mathbb{B}_d$ joining $z$ and $w$} \}.
  \end{equation*}
  Moreover, \cite[Theorem 2]{MPS85} (see also \cite[Proposition 9]{ARS+11}) shows that the right-hand side
  remains unchanged when taking the infimum over all continuous curves, hence $\beta = \rho^*$.

  (b) We apply \eqref{eqn:MPS} to the Dirichlet kernel $k$.
  A simple computation shows that for $z \neq 0$, we have
  \begin{equation*}
    \frac{\partial^2 \log k(z,z)}{\partial z \partial \overline{z}}
    = \frac{\partial^2}{\partial z \partial \overline{z}} \log \log \Big( \frac{1}{1 - |z|^2} \Big) = g(z).
  \end{equation*}
  Direct inspection of the Taylor series of $\log(1-z)$ at the origin shows that $g$ extends continuously
  to $0$ with $g(0) = \frac{1}{2}$, so the above identity holds for all $z \in \mathbb{D}$,
  and the statement follows from \eqref{eqn:MPS}.
\end{proof}

\begin{remark}
  \label{rem:easy_pseudohyperbolic}
  In Lemma \ref{lem:metrics} (a), we will actually only need the lower bound $\beta \le \rho^*$.
  This lower bound can be proved in a more elementary way as follows.
  Since $\rho = \tanh(\beta)$ and since the derivative of $\tanh$ at $0$ equals $1$,
  a simple estimate shows that $\ell_{\rho}(\gamma) = \ell_{\beta}(\gamma)$
  for any continuous curve $\gamma$, see \cite[Lemma 2.5.2 (a)]{JP13}. So if $\gamma$
  is a continuous curve joining $z$ and $w$, then $\beta(z,w) \le \ell_{\beta}(\gamma) = \ell_\rho(\gamma)$,
  which gives the inequality $\beta \le \rho^*$.
\end{remark}

\section{Proof of main result}
\label{sec:proof}

\subsection{From surjective homomorphisms to bi-Lipschitz embeddings}

In the first step towards the proof of Theorem \ref{T:main},
we show that any surjective homomorphism as in Theorem \ref{T:main}
induces a bi-Lipschitz embedding of $(\mathbb{D},\delta_{\mathcal{D}})$
into $(\mathbb{B}_d,\delta_{H^2_d})$.
We need the following characterization of the metric $\delta_{\mathcal{H}}$ in terms
of the multiplier algebra for complete Pick spaces.
It applies in particular to the Dirichlet space and to the Drury-Arveson space.

\begin{lemma}
  \label{lem:metric_pick}
  Let $\mathcal{H}$ be a complete Pick space on $X$. Then
  \begin{equation*}
    \delta_{\cH}(z,w) = \sup \{|\varphi(z)| : \varphi(w) = 0 , \|\varphi\|_{\Mult(\cH)} \le 1 \}
  \end{equation*}
  for all $z,w \in X$. Moreover, the supremum is attained.
\end{lemma}

\begin{proof}
  The result is well known. For the convenience of the reader, we provide a short argument.
  Let $z,w \in X$ and let $\lambda \in \mathbb{C}$. By the Pick property, there exists
  $\varphi \in \Mult(\mathcal{H})$ with $\|\varphi\|_{\Mult(\mathcal{H})} \le 1$,
  $\varphi(w) = 0$ and $\varphi(z) = \lambda$ if and only if the Pick matrix
  \begin{equation*}
    \begin{bmatrix}
      K(z,z) (1 - |\lambda|^2) & K(z,w) \\ K(w,z) & K(w,w)
    \end{bmatrix}
  \end{equation*}
  is positive semi-definite. By Sylvester's criterion, this happens if and only
  if $|\lambda| \le 1$ and the determinant of the matrix is non-negative,
  which after rearranging is seen to be equivalent to
  \begin{equation*}
    |\lambda| \le \Big( 1 - \frac{|K(z,w)|^2}{K(z,z) K(w,w)} \Big)^{1/2}.
  \end{equation*}
  Since the right-hand side equals $\delta_{\mathcal{H}}(z,w)$, the result follows.
\end{proof}

If $E$ is a Banach space, we denote the closed unit ball of $E$ by $B_E$.
If $T: E \to F$ is a surjective continuous linear operator between Banach spaces,
let
\begin{equation*}
  q(T) = \sup \{r \ge 0 : T(B_E) \supset r B_F \}
\end{equation*}
be the surjectivity modulus of $T$. By the open mapping theorem, $q(T) > 0$.
If $\widetilde{T}: E / \ker(T) \to F$ denotes the induced operator, then $q(T) = \| \widetilde{T}^{-1}\|^{-1}$.

The following result generalizes \cite[Theorem 6.2]{DHS15}.
The Lipschitz constant obtained below is an improvement of the Lipschitz constant obtained there
by a factor of $2$.

\begin{proposition}
  \label{prop:iso_bi_lipschitz}
  Let $\mathcal{H}$ and $\mathcal{K}$ be complete Pick spaces on $X$ and $Y$, respectively.
  Let $F: Y \to X$ be a mapping that induces a homomorphism
  \begin{equation*}
    \Phi: \Mult(\mathcal{H}) \to \Mult(\mathcal{K}), \quad \varphi \mapsto \varphi \circ F.
  \end{equation*}
\begin{enumerate}[label=\normalfont{(\alph*)},wide]
    \item The homomorphism $\Phi$ is continuous, and $F$ is Lipschitz with
      \begin{equation*}
        \delta_{\mathcal{H}}(F(z),F(w)) \le \|\Phi\| \delta_{\mathcal{K}}(z,w) \quad \text{ for all } z,w \in Y.
      \end{equation*}
    \item If $\Phi$ is surjective, then $q(\Phi) > 0$ and $F$ is bi-Lipschitz with
      \begin{equation*}
        q(\Phi) \delta_{\mathcal{K}} (z,w)  \le  \delta_{\mathcal{H}}(F(z),F(w)) \le \|\Phi\| \delta_{\mathcal{K}}(z,w) \quad \text{ for all } z,w \in Y.
      \end{equation*}
  \end{enumerate}
\end{proposition}

\begin{proof}
  (a)
  Since multiplier algebras are commutative semi-simple Banach algebras, the homomorphism
  $\Phi$ is continuous by a standard automatic continuity result;
  see, for instance, \cite[Corollary 2.1.10]{Kaniuth09}.
  Let $C > \|\Phi\|$ and let $z,w \in Y$.
  By Lemma \ref{lem:metric_pick}, there exists $\varphi \in \Mult(\mathcal{H})$
  of norm at most $1$ such that $\varphi(F(z)) = 0$ and $\varphi(F(w)) = \delta_{\mathcal{H}}(F(z),F(w))$.
  Let $\psi = C^{-1} (\varphi \circ F)$. Then $\psi$ belongs to the unit ball of $\Mult(\mathcal{K})$
  and $\psi(z) = 0$ and $\psi(w) = C^{-1} \delta_{\mathcal{H}}(F(z),F(w))$.
  Another application of Lemma \ref{lem:metric_pick} shows that
  \begin{equation*}
    \delta_{\mathcal{K}}(z,w) \ge C^{-1} \delta_{\mathcal{H}}(F(z),F(w)).
  \end{equation*}
  Since $C > \|\Phi\|$ was arbitrary, the result follows.

  (b) Suppose that $\Phi$ is surjective. By the open mapping theorem, $q(T) > 0$.
  Let $z,w \in Y$. By Lemma \ref{lem:metric_pick}, there exists $\psi \in \Mult(\mathcal{K})$
  of norm at most $1$
  with $\psi(z) = 0$ and $\psi(w) = \delta_{\mathcal{K}}(z,w)$.
  Let $0 <\varepsilon < q(T)$.
  By definition of $q(\Phi)$, there exists
  $\varphi \in \Mult(\mathcal{H})$ of norm at most $1$ such that $\varphi \circ F = \varepsilon \psi$.
  Then $\varphi(F(z)) = 0$ and $\varphi(F(w)) = \varepsilon \delta_{\mathcal{K}}(z,w)$,
  so Lemma \ref{lem:metric_pick} shows that
  \begin{equation*}
    \delta_{\mathcal{H}}(F(z),F(w)) \ge \varepsilon \delta_{\mathcal{K}}(z,w).
  \end{equation*}
  Since $\varepsilon < q(T)$ was arbitrary, the first inequality follows.
  The second inequality was already established in (a).
\end{proof}

The following consequence is immediate from part (b) of Proposition \ref{prop:iso_bi_lipschitz}.

\begin{corollary}
  \label{cor:bi_Lipschitz}
  Let $d \in \mathbb{N}$.
  If $b: \mathbb{D} \to \mathbb{B}_d$ induces a surjective homomorphism
  \begin{equation*}
    \Mult(H^2_d) \to \Mult(\mathcal{D}), \quad \varphi \mapsto \varphi \circ b,
  \end{equation*}
  then $b$ is a bi-Lipschitz map from $(\mathbb{D},\delta_{\mathcal{D}})$ into $(\mathbb{B}_d, \delta_{H^2_d})$.
  \qed
\end{corollary}

We will show that no such bi-Lipschitz map exists.

\subsection{Lipschitz maps into the ball}

In the sequel, we will simply write $\delta = \delta_{\mathcal{D}}$ for the metric on $\mathbb{D}$ induced by the Dirichlet space.
We also continue to write $\rho = \delta_{H^2_d}$ for the pseudo-hyperbolic metric on $\mathbb{B}_d$.

Our next goal is to show that if $f: (\mathbb{D},\delta) \to (\mathbb{B}_d, \rho)$ is Lipschitz with $f(0) = 0$,
then $\|f(z)\|$ can only approach $1$ very slowly as $|z|$ approaches $1$;
that is, we seek upper bounds on $\|f(z)\|$.
For isometric maps $f$, this was done by Rochberg \cite{Rochberg16} by noting
that $\|f(z)\| = \rho(0, f(z)) = \delta(0,z)$ in this case and estimating $1 - \delta(0,z)$.
In the bi-Lipschitz setting, the upper bound on $\|f(z)\|$ that can be deduced from this argument
takes the form $\|f(z)\| \le C \delta(0,z)$ for some constant $C$.
This is not a very useful estimate if $|z|$ is close to $1$, because
we always have the trivial upper bound $\|f(z)\| \le 1$.
Instead, we will use lengths of curves as defined in Subsection \ref{ss:length}
to obtain a better upper bound on $\|f(z)\|$.

If $f,g: [0,1) \to [0,\infty)$ are two functions, we will write
$f \sim g$ as $r \to 1$ if $f$ and $g$ do not vanish near $1$ and $\lim_{r \to 1}
\frac{f(r)}{g(r)} = 1$.

\begin{lemma}
  \label{lem:Dirichlet_length}
  Let $0 \le r <1$ and let $\gamma_r: [0,r] \to \mathbb{D}, t \mapsto t$. Then
  \begin{equation*}
    \ell_{\delta}(\gamma_r) \sim \Big( \log \Big( \frac{1}{1 - r} \Big) \Big)^{1/2} \quad \text{ as } r \to 1.
  \end{equation*}
  Hence, there exists a constant $M \in (0,\infty)$ such that
  \begin{equation*}
    \ell_\delta(\gamma_r) \le M \Big( \log \Big(\frac{1}{1-r} \Big) \Big)^{1/2} \quad \text{ for all } r \in [0,1).
  \end{equation*}
\end{lemma}

\begin{proof}
  Let $g: \mathbb{D} \to [0,\infty)$ be the continuous function appearing in part (b) of Lemma \ref{lem:metrics}.
  By that lemma, $\ell_{\delta}(\gamma_r) = \int_{0}^r g(t)^{1/2} \, dt$, hence
  \begin{equation*}
    \frac{d \ell_{\delta}(\gamma_r)}{d r} = g(r)^{1/2}
    \sim \frac{1}{\big( \log \big(\frac{1}{1 - r^2} \big) \big)^{1/2} (1 - r^2)}
    \sim \frac{1}{2 \big( \log \big(\frac{1}{1 - r} \big) \big)^{1/2} (1-r)}
  \end{equation*}
  as $r \to 1$. On the other hand,
  \begin{equation*}
    \frac{d}{dr} 
    \Big( \log \Big( \frac{1}{1 - r} \Big) \Big)^{1/2} = 
    \frac{1}{2 \big( \log \big(\frac{1}{1 - r} \big) \big)^{1/2} (1-r)}.
  \end{equation*}
  The first statement now follows from L'H\^opital's rule.

  To deduce the second statement from the first one, it suffices to show that
  the desired estimate holds for fixed $r_0 < 1$ and all $r \in [0,r_0]$.
  But if $r_0 <1$, then $g$ is bounded on $[0,r_0]$.
  So if $M > 0$ is such that $|g(t)|^{1/2} \le M$ for all $t \in [0,r_0]$, then
  for all $r \in [0,r_0]$, we have
  \begin{equation*}
    \ell_{\delta}(\gamma_r) \le M r \le M r^{1/2} \le M \Big( \log \Big( \frac{1}{1-r} \Big) \Big)^{1/2}.
  \end{equation*}
  as desired.
\end{proof}

The following proposition now shows that for any Lipschitz map $f: (\mathbb{D},\delta) \to (\mathbb{B}_d,\rho)$
with $f(0) = 0$, the quantity $\|f(z)\|$ can only approach $1$ very slowly as $|z| \to 1$.

\begin{proposition}
  \label{prop:Lipschitz_slow_growth}
  Let $f: (\mathbb{D},\delta) \to (\mathbb{B}_d,\rho)$ be a Lipschitz mapping with $f(0) = 0$.
  Then there exists a constant $C \in (0,\infty)$ such that
  \begin{equation*}
    \|f(z)\| \le 1 - \exp \Big( - C \Big( \log \Big( \frac{1}{1 - |z|} \Big) \Big)^{1/2} \Big) \quad \text{ for all } z \in \mathbb{D}.
  \end{equation*}
  In particular, for any $\alpha > 0$, there exists $r_0 \in (0,1)$ such that
  \begin{equation*}
    \|f(z)\| \le 1 - (1 - |z|)^\alpha
  \end{equation*}
  for all $z \in \mathbb{D}$ with $|z| \ge r_0$.
\end{proposition}

\begin{proof}
  Since $f$ is Lipschitz, there exists $L \in (0,\infty)$ such that
  $\rho(f(z),f(w)) \le L \delta(z,w)$ for all
  $z,w \in \bD$. It easily follows from the definition of the length of a curve that
  \begin{equation*}
    \ell_\rho (f \circ \gamma) \le L \ell_\delta(\gamma)
  \end{equation*}
  for every curve $\gamma: [a,b] \to \bD$.
  
  Now, let $z \in \mathbb{D}$, let $r = |z|$ and write $z = \lambda r$
  for $\lambda \in \mathbb{C}$ with $|\lambda|=1$.
  Let $\gamma_r: [0,r] \to \mathbb{D}, t \mapsto t$.
  Then $f \circ( \lambda \gamma_r)$ is a curve in $\mathbb{B}_d$ from $0$
  to $f(z)$, hence
  Lemma \ref{lem:metrics} (a), see also Remark \ref{ss:length}, and rotation invariance of the metric $\delta$ imply that
  \begin{equation*}
    \beta(0,f(z)) = \rho^*(0,f(z)) \le \ell_{\rho}(f \circ (\lambda \gamma_r))
    \le L \ell_{\delta} (\lambda \gamma_r)
    = L \ell_{\delta}(\gamma_r).
  \end{equation*}
  Estimating the right-hand side with the help of Lemma \ref{lem:Dirichlet_length}
  and recalling the definition of $\beta$, we find that
  \begin{equation*}
    \log \Big( \frac{1}{1 - \|f(z)\|} \Big) \le 2 \beta(0,f(z)) \le 2 L M \Big( \log \Big( \frac{1}{1-r}\Big) \Big)^{1/2}.
  \end{equation*}
  Rearranging this inequality gives the first statement with $C = 2 LM$.

  As for the second statement, let $\alpha > 0$ and choose $r_0 \in (0,1)$ so that
  $C \le \alpha ( \log(\frac{1}{1-r_0} ))^{1/2}$. If $|z| \ge r_0$, then the first estimate yields
  \begin{align*}
    \|f(z)\| \le 1 - \exp \Big( - C \Big( \log \Big( \frac{1}{1 - |z|} \Big) \Big)^{1/2} \Big)
    &\le 1- \exp \Big( - \alpha \log \Big( \frac{1}{1 - |z|} \Big) \Big) \\ &= 1 - (1 - |z|)^\alpha,
  \end{align*}
  which completes the proof.
\end{proof}

\subsection{Separated sets}

So far, we have only used the Lipschitz property of our embeddings from $\mathbb{D}$ into $\mathbb{B}_d$.
To make use of the fact that they are also bounded below, we will consider separated
sets in the two metric spaces. This is again similar to Rochberg's arguments in \cite{Rochberg16}
in the isometric setting, but some modifications are necessary.

Let $(X,d)$ be a metric space and let $\varepsilon > 0$.
We say that a subset $D \subset X$ is \emph{$\varepsilon$-separated}
if any two distinct points in $D$ have distance at least $\varepsilon$,
i.e.\ $d(x,y) \ge \varepsilon$ for all $x,y \in D$ with $x \neq y$.
Given a subset $S \subset X$ and $\varepsilon > 0$, we let
\begin{equation*}
  N(S,d,\varepsilon) = \sup \{ |D| : D \subset S \text { is } \varepsilon \text{-separated} \}.
\end{equation*}

These notions are useful for studying mappings that are bounded below because of the following
obvious lemma.

\begin{lemma}
  \label{lem:bounded_below}
  Let $(X,d_X)$ and $(Y,d_Y)$ be metric spaces and
  let $f: X \to Y$. Suppose that there exists a constant $m > 0$ such that
  $d_Y(f(x),f(y)) \ge m d_X(x,y)$ for all $x,y \in X$.
  Then, for any $S \subset X$, we have
  \begin{equation*}
    N(f(S),d_Y,m \varepsilon) \ge N(S,d_X,\varepsilon).
  \end{equation*}
\end{lemma}

\begin{proof}
  If $D \subset S$ is $\varepsilon$-separated, then $f(D) \subset f(S)$
  is $m \varepsilon$-separated.
\end{proof}

Let $C_r$ denote the circle of radius $r$ around $0$ in the complex plane.
Our next goal is to establish a lower bound for $N(C_r,\delta,\varepsilon)$, that is,
we wish to find large subsets of $C_r$ that are $\varepsilon$-separated in the
$\delta$-metric. The following lemma reduces the task of checking whether certain subsets
of $C_r$ are $\varepsilon$-separated to checking distances between adjacent points.

\begin{lemma}
  \label{lem:circle_increasing}
  Let $0 < r < 1$, let
  $0 \le \theta_1 < \theta_2 < \ldots < \theta_n \le \pi$ and let
  \begin{equation*}
    D = \{r e^{ i \theta_1} , \ldots, r e^{ i \theta_n}\}.
  \end{equation*}
  Then $D$ is $\varepsilon$-separated with respect to $\delta$
  if and only if $\delta(r e^{ i \theta_k}, r e^{i \theta_{k+1}}) \ge \varepsilon$
  for all $1 \le k \le n-1$.
\end{lemma}

\begin{proof}
  Necessity is clear. To prove sufficiency, by rotation invariance of the metric $\delta$,
  it suffices to show that
  \begin{equation*}
    [0,\pi] \to [0,\infty), \quad t \mapsto \delta(r e^{ i t}, r) ,
  \end{equation*}
  is increasing.
  From \eqref{eqn:Dirichlet_metric}, we see that
  \begin{equation*}
    1 - \delta (r e^{ i t}, r)^2 = \frac{|\log(1 -r^2 e^{i t})|^2}{\log(1 - r^2)^2},
  \end{equation*}
  so
  it even suffices to show that for each $s \in (0,1)$, the function
  \begin{equation*}
    f: [0, \pi] \to [0,\infty), \quad t \mapsto | \log(1 - s e^{i t})|^2,
  \end{equation*}
  is decreasing.
  
  Writing $f(t) = \log(1 - s e^{i t}) \log(1 - s e^{- i t})$,
  a straightforward calculation shows that
  \begin{equation*}
    f'(t) = \frac{2 s}{|1 -s e^{i t}|^2} \Im ( \log (1-s e^{ it}) (s - e^{ - it})).
  \end{equation*}
  We finish the proof by showing that for each $t \in [0, \pi]$, the function
  \begin{equation*}
    h: [0,1) \to \mathbb{R}, \quad
    s \mapsto \Im ( \log(1 - s e^{i t}) (s - e^{ - it }))
  \end{equation*}
  is bounded above by $0$.
  To this end, notice that $h(0) = 0$,
  and by another small computation,
  \begin{equation*}
    h'(s) = \Im ( \log (1 - s e^{ it})) \le 0
  \end{equation*}
  for all $t \in [0,\pi]$ and all $s \in [0,1)$,
  because $1- s e^{ it}$ belongs to the closed lower half plane for these
  values of $s$ and $t$.
  Hence, $h$ is decreasing in $s$, so $h(s) \le 0$
  for all $s \in [0,1)$, as asserted.
\end{proof}

The following lemma will allow us to place sufficiently many $\varepsilon$-separated points on $C_r$ as $r$ approaches $1$.

\begin{lemma}
  \label{lem:circle_asymptotic}
  For $0 \le r < 1$, let $\theta(r) = \sqrt{1 - r}$. Then
  \begin{equation*}
    \lim_{r \to 1} \delta (r e^{i \theta(r)}, r) = \sqrt{\frac{3}{4}}.
  \end{equation*}
\end{lemma}

\begin{proof}
  Note that
  \begin{equation*}
    1 - \delta (r e^{i \theta(r)},r)^2 = \frac{ |\log(1 - r^2 e^{i \theta(r)})|^2}{\log(1-r^2)^2},
  \end{equation*}
  so we have to show that
  \begin{equation}
    \label{eqn:circle_to_show}
    \lim_{r \to 1} \frac{|\log(1 - r^2 e^{ i \theta(r)})|}{- \log(1-r^2)} = \frac{1}{2}.
  \end{equation}
  Since the imaginary part of $\log$ is bounded, we see that
  \begin{align}
    \label{eqn:circle_inter}
    \lim_{r \to 1} \frac{|\log(1 - r^2 e^{ i \theta(r)})|}{- \log(1-r^2)}
    = \lim_{r \to 1} \frac{| \Re \log(1 - r^2 e^{ i \theta(r)})|}{- \log(1-r)}
    = \lim_{r \to 1} \frac{\log|1 - r^2 e^{ i \theta(r)}|^2}{2 \log(1-r)}.
  \end{align}
  The Taylor series expansion of cosine shows that
  \begin{equation*}
    r \mapsto |1 - r^2 e^{i \theta(r)}|^2 = 1 - 2 \cos(\sqrt{1-r}) r^2 + r^4
  \end{equation*}
  extends to a function $h$ that is differentiable at $r=1$ and satisfies $h(1) = 0$
  and $h'(1) = -1$.
  Thus, $|1  - r^2 e^{i \theta(r)}|^2/(1-r)$ converges to $1$ as $r \to 1$,
  and so
  \begin{equation*}
    \lim_{r \to 1} \frac{\log|1 - r^2 e^{ i \theta(r)}|^2}{2 \log(1-r)}=
    \lim_{r \to 1} \frac{\log \frac{|1 - r^2 e^{ i \theta(r)}|^2}{1 - r} + \log(1-r)}{2 \log(1-r)} = \frac{1}{2}.
  \end{equation*}
  In combination with \eqref{eqn:circle_inter}, this proves \eqref{eqn:circle_to_show}.
\end{proof}

We are now able to establish the desired lower bound for $N(C_r,\delta,\varepsilon)$.
\begin{lemma}
  \label{lem:points_on_circle}
  Let $C_r = \{z \in \mathbb{C}: |z|=r\}$. If $\varepsilon < \sqrt{3/4}$, then
  there exists $r_0 < 1$ such that
  \begin{equation*}
    N(C_r,\delta,\varepsilon) \ge \frac{1}{\sqrt{1-r}}
  \end{equation*}
  for all $r_0 < r < 1$.
\end{lemma}

\begin{proof}
  Let $\theta(r) = \sqrt{1 - r}$, let $N(r) = \lfloor \frac{\pi}{\sqrt{1-r}} \rfloor$
  and let
  \begin{equation*}
    D(r) = \{ r, r e^{ i \theta(r)}, \ldots, r e^{ i N(r) \theta(r)} \}.
  \end{equation*}
  Observe that $D(r) \subset C_r$ and $|D(r)| \ge \frac{\pi}{\sqrt{1 - r}}$.

  By Lemma \ref{lem:circle_asymptotic}, there exists $r_0 <1$ such that
  \begin{equation*}
    \delta(r e^{ i \theta(r)}, r) \ge \varepsilon \quad \text{ for all } r \in [r_0,1).
  \end{equation*}
  Rotational invariance of $\delta$ and Lemma \ref{lem:circle_increasing}
  now show that $D(r)$ is $\varepsilon$-separated for $r \ge r_0$.
  Therefore,
  \begin{equation*}
    N(C_r,\delta,\varepsilon) \ge |D(r)| \ge \frac{1}{\sqrt{1-r}}. \qedhere
  \end{equation*}
\end{proof}

Finally, we require the following upper bound for the number of points in a ball
that are separated with respect to the pseudohyperbolic metric. The result is due to Duren and Weir;
see \cite[Lemma 5]{DW07}.

\begin{lemma}
  \label{lem:points_in_ball}
  Let $d \in \mathbb{N}$, let $0 < r < 1$ and let $B_r = \{z \in \mathbb{C}^d: \|z\| \le r\}$.
  Then
  \begin{equation*}
    N(B_r,\rho,\varepsilon) \le \Big(\frac{2}{\varepsilon} + 1 \Big)^{2 d} \frac{1}{(1-r^2)^d}
  \end{equation*}
  for all $\varepsilon > 0$. \qed
\end{lemma}

\subsection{Proof of main result}

We are now in position to prove the main result.
In light of Corollary \ref{cor:bi_Lipschitz},
the following result implies Theorem \ref{T:main}.

\begin{theorem}
  \label{thm:no_embedding}
  For any finite $d$ there does not exist a bi-Lipschitz map from $(\bD,\delta_{\mathcal{D}})$ into $(\bB_d, \delta_{H^2_d})$.
\end{theorem}

\begin{proof}
  Suppose towards a contradiction that there exists a bi-Lipschitz map $f: \bD \to \bB_d$ for some $d< \infty$.
  Since biholomorphic automorphisms of $\bB_d$ are isometries with respect to the pseudohyperbolic metric,
  we may assume without loss of generality that $f(0)= 0$.

  As before, let $C_r = \{z \in \mathbb{C}: |z| = r\}$ and
  $B_s = \{z \in \mathbb{C}^d : \|z\| \le s\}$.
  Applying Proposition \ref{prop:Lipschitz_slow_growth}
  with $\alpha = 1/ (2 d + 1)$
  we find $r_0 < 1$ such that
  $f(C_r)$ is contained in $B_{s(r)}$ for all $r_0 < r < 1$, where
  \begin{equation*}
    s(r) =  1 - (1 - r)^{1/(2 d+1)}.
  \end{equation*}
  Lemma \ref{lem:points_on_circle} shows that by increasing $r_0$ if necessary, we may find $\varepsilon > 0$ such that
  \begin{equation*}
    N(C_r,\delta,\varepsilon) \ge \frac{1}{\sqrt{1 -r }} \quad \text{ for all } r_0 < r < 1.
  \end{equation*}
  Assuming that $f$ is bounded below by $m$, Lemma \ref{lem:bounded_below} then implies that
  \begin{equation}
    \label{eqn:main_1}
    \frac{1}{\sqrt{1-r}}
    \le N(B_{s(r)},\delta_{H^2_d},m \varepsilon) \quad \text{ for all } r_0 < r < 1.
  \end{equation}
  On the other hand, Lemma \ref{lem:points_in_ball} shows that
  \begin{equation}
    \label{eqn:main_2}
    N(B_{s(r)},\delta_{H^2_d},m \varepsilon) \le
    \Big(\frac{2}{m \varepsilon} + 1 \Big)^{2 d}
      \frac{1}{(1-s(r)^2)^d} \le C \frac{1}{(1-r)^{d / (2 d+1)}}
  \end{equation}
  for all $0 < r < 1$ and some constant $C < \infty$ that does not depend on $r$.
  Combining \eqref{eqn:main_1} and \eqref{eqn:main_2}, we arrive at a contradiction.
  Hence, there does not exist a bi-Lipschitz map from $(\mathbb{D},\delta)$ into $(\mathbb{B}_d,\rho)$
  for $d < \infty$.
\end{proof}

\subsection{Weighted Dirichlet spaces}
For $a \in (0,1)$, let $\mathcal{D}_a$ be the reproducing kernel Hilbert space on $\mathbb{D}$
with kernel
\begin{equation*}
  k_a(z,w) = \frac{1}{(1 - z \overline{w})^a}.
\end{equation*}
These spaces are weighted Dirichlet spaces.
It is well known that they are also complete Pick spaces, which follows from the fact
that the power
series coefficients of $1-1/k_a$ are non-negative \cite[p.~22]{MS94a}.
It is natural to ask if Theorem \ref{T:main} can be extended to these spaces.

\begin{question}
  \label{quest:D_a}
  Let $a \in (0,1)$. Do there exist $d \in\mathbb{N}$ and a map $b: \mathbb{D} \to \mathbb{B}_d$
  that induces a surjective homomorphism
  \begin{equation*}
    \Mult(H^2_d) \to \Mult(\mathcal{D}_a), \quad \varphi \mapsto \varphi \circ b?
  \end{equation*}
\end{question}

It was shown in \cite[Corollary 11.9]{Hartz17a} that for any $a \in (0,1)$, there does
not exist a map $b: \mathbb{D} \to \mathbb{B}_d$ with $d< \infty$ such that
\begin{equation*}
  k_a(z,w) = \frac{1}{1 - \langle b(z),b(w) \rangle }.
\end{equation*}

The proof of Theorem \ref{T:main} given here does not generalize to $\mathcal{D}_a$.
In fact, Theorem \ref{thm:no_embedding} is not true with $\mathcal{D}_a$
in place of $\mathcal{D}$, as the metric $\delta_{\mathcal{D}_a}$ is equivalent
to the pseudo-hyperbolic metric $\rho$ on $\mathbb{D}$.
Indeed, from Equation \eqref{eqn:metric}, it follows that
\begin{equation*}
  a^{1/2} \rho \le \delta_{\mathcal{D}_a} = \Big( 1 - ( 1 - \rho^2)^a \Big)^{1/2} \le \rho.
\end{equation*}
Thus, the identity mapping $(\mathbb{D}, \delta_{\mathcal{D}_a}) \to (\mathbb{D}, \rho)$ is bi-Lipschitz.
Nonetheless, the multiplier algebras of $H^2$ and of $\mathcal{D}_a$ do not coincide.

It therefore appears that different arguments are needed to answer Question \ref{quest:D_a}.

\section{Other notions of embedding}
\label{sec:other}

The following result shows that Theorem \ref{T:main} also rules out
the existence of other types of embedding for the Dirichlet space.
We also relate Theorem \ref{T:main} to the point of view taken for instance in \cite{DHS15} and \cite{SS14}.
To this end, let $V \subset \mathbb{B}_d$ and define $H^2_d \big|_V$ to be the reproducing
kernel Hilbert space on $V$ whose reproducing kernel is the restriction of the Drury--Arveson kernel to $V \times V$.
Equivalently,
\begin{equation*}
  H^2_d \big|_V = \{ f \big|_V: f \in H^2_d \},
\end{equation*}
equipped with the quotient norm.

\begin{proposition}
  \label{prop:other_emb}
  Let $\mathcal{H}$ be a reproducing kernel Hilbert space
  on a set $X$ separating the points of $X$ with reproducing kernel $k$.
  Let $d \in \mathbb{N} \cup \{\infty\}$ and let $b: X \to \mathbb{B}_d$ and
  $\delta: X \to \mathbb{C} \setminus \{0\}$ be mappings.
  Let $V= b(X)$.

  Among the following statements, the implications
  \begin{center}
  \normalfont{(i)} $\Leftrightarrow$
  \normalfont{(i')} $\Leftrightarrow$
  \normalfont{(i'')} $\Rightarrow$
  \normalfont{(ii)} $\Leftrightarrow$
  \normalfont{(ii')} $\Rightarrow$
  \normalfont{(iii)} $\Leftrightarrow$
  \normalfont{(iii')}
  \end{center}
  hold.

  \begin{enumerate}
    \item[\normalfont{(i)}] $k(z,w) = \frac{\delta(z) \overline{\delta(w)}}{1 - \langle b(z), b(w) \rangle}$
      for all $z, w \in X$;
    \item[\normalfont{(i')}] the assignment
      \begin{equation*}
        H^2_d \to \mathcal{H}, \quad f \mapsto \delta \cdot (f \circ b),
      \end{equation*}
      defines a co-isometry;
    \item[\normalfont{(i'')}] the assignment
      \begin{equation*}
        H^2_d \big|_V \to \mathcal{H}, \quad f \mapsto \delta \cdot (f \circ b),
      \end{equation*}
      defines a unitary;
    \item[\normalfont{(ii)}] the assignment
      \begin{equation*}
        H^2_d  \to \mathcal{H}, \quad f \mapsto \delta \cdot (f \circ b),
      \end{equation*}
      defines a surjection;
    \item[\normalfont{(ii')}] the assignment
      \begin{equation*}
        H^2_d \big|_V \to \mathcal{H}, \quad f \mapsto \delta \cdot (f \circ b),
      \end{equation*}
      defines a bijection;
    \item[\normalfont{(iii)}] the assignment
      \begin{equation*}
        \Mult(H^2_d)  \to \Mult(\mathcal{H}), \quad \varphi \mapsto \varphi \circ b,
      \end{equation*}
      defines a surjection;
    \item[\normalfont{(iii')}] the assignment
      \begin{equation*}
        \Mult(H^2_d \big|_V)  \to \Mult(\mathcal{H}), \quad \varphi \mapsto \varphi \circ b,
      \end{equation*}
      defines a bijection.
  \end{enumerate}
\end{proposition}

\begin{proof}
  The proof consists of routine arguments with reproducing kernels. We sketch the main ideas.

  (i) $\Rightarrow$ (i') Let $K(z,w) = \frac{1}{1 - \langle z,w \rangle }$ be the reproducing kernel for $H^2_d$.
  Condition (i) implies that
  \begin{equation*}
    \langle k(\cdot,w), k(\cdot,z) \rangle_{\mathcal{H}} = \langle \overline{\delta(w)} K(\cdot,b(w)), \overline{\delta(z)} K(\cdot,b(z)) \rangle_{H^2_d} \quad (z,w \in X),
  \end{equation*}
  from which it follows that there exists an isometry $V: \mathcal{H} \to H^2_d$ with
  \begin{equation}
    \label{eqn:iso}
    V k(\cdot,w) = \overline{\delta(w)} K(\cdot,b(w)) \quad (w \in X).
  \end{equation}
  The adjoint $V^*: H^2_d \to \mathcal{H}$ is the map in the statement of (i').

  (i') $\Rightarrow$ (i)
  If $T$ denotes the co-isometry in the statement of (i'), then $V = T^*$
  is an isometry satisfying \eqref{eqn:iso}, from which (i) follows.

  (i') $\Leftrightarrow$ (i'') The restriction map $H^2_d \to H^2_d \big|_V$ is a co-isometry
  whose kernel is $I(V)$, the space of all functions in $H^2_d$ vanishing on $V$.
  So if the map in (i'') is unitary, then the map in (i') is a co-isometry.
  Conversely, if the map in (i') is a co-isometry, then its kernel is $I(V)$, so the map in (i'')
  is unitary.

  (i') $\Rightarrow$ (ii) is trivial.

  (ii) $\Leftrightarrow$ (ii') follows in the same way as (i') $\Leftrightarrow$ (ii').

  (ii') $\Rightarrow$ (iii') Let
  \begin{equation*}
    T: H^2_d \big|_V \to \mathcal{H}, \quad f \mapsto \delta \cdot (f \circ b),
  \end{equation*}
  be the map in (ii'). By the closed graph theorem and the open mapping theorem, $T$ is bounded
  and has a bounded inverse. Moreover, since $\mathcal{H}$ separates the points of $X$,
  surjectivity of $T$ shows that $b$ is injective, hence the inverse of $T$ is given by
  \begin{equation*}
    T^{-1} g = \Big( \frac{g}{\delta} \Big) \circ b^{-1}.
  \end{equation*}
  Therefore, if $\varphi \in \Mult(H^2_d \big|_V)$, then
  \begin{equation*}
    T M_\varphi T^{-1} g = (\varphi \circ b) \cdot g \quad (g \in \mathcal{H}),
  \end{equation*} 
  so the operator $T M_\varphi T^{-1}$ on $\mathcal{H}$ is given by multiplication with $\varphi \circ b$.
  In particular, $\varphi \circ b \in \Mult(\mathcal{H})$.
  Similarly, if $\psi \in \Mult(\mathcal{H})$, then $T^{-1} M_\psi T$
  is the operator of multiplication by $\psi \circ b^{-1}$ on $H^2_d \big|_V$,
  so $\psi \circ b^{-1} \in \Mult(H^2_d \big|_V)$. Hence (iii') holds.

  (iii) $\Leftrightarrow$ (iii') The complete Pick property of $H^2_d$ shows that
  the restriction map $\Mult(H^2_d) \to \Mult(H^2_d \big|_V)$ is surjective,
  and its kernel consists of all multipliers vanishing on $V$.
  The equivalence of (iii) and (iii') readily follows from this fact.
\end{proof}

In \cite[Section 7.5]{SS14}, Salomon and Shalit asked
if there exist $d \in \mathbb{N}$ and $V \subset \mathbb{B}_d$ such that $\Mult(\mathcal{D})$ is isomorphic to $\Mult(H^2_d \big|_V)$.
Using Theorem \ref{T:main}, we can also show that this cannot happen.

\begin{theorem}
  \label{thm:no_iso}
  There do not exist $d \in \mathbb{N}$ and a subset $V \subset \mathbb{B}_d$
  such that $\Mult(\mathcal{D})$ is algebraically isomorphic to $\Mult(H^2_d \big|_V)$.
\end{theorem}

\begin{proof}
  We use arguments from the study of the isomorphism problem for multiplier
  algebras of complete Pick spaces to show that any isomorphism
  must be given by composition with a map $b: \mathbb{D} \to \mathbb{B}_d$;
  see for instance \cite[Theorem 2.4]{DHS15}.

  Suppose towards a contradiction that $V \subset \mathbb{B}_d$ for finite $d$
  and that $\Phi: \Mult(H^2_d \big|_V) \to \Mult(\mathcal{D})$
  is an algebraic isomorphism. Let
  \begin{equation*}
    \Psi: \Mult(H^2_d) \to \Mult(\mathcal{D}), \quad \varphi \mapsto \Phi( \varphi \big|_V).
  \end{equation*}
  The complete Pick property of $H^2_d$ and surjectivity of $\Phi$ imply that $\Psi$ is a
  surjective unital homomorphism.
  We consider the adjoint $\Psi^*$ between the maximal ideal spaces $\mathcal{M}(\Mult(\mathcal{D}))$
  and $\mathcal{M}(\Mult(H^2_d))$. For each $\lambda \in \mathbb{D}$, the character
  of point evaluation $\delta_\lambda$ belongs to $\mathcal{M}(\Mult(\mathcal{D}))$,
  and similarly for $\mathcal{M}(\Mult(H^2_d))$. Moreover, we have a map
  \begin{equation*}
    \pi: \mathcal{M}(\Mult(H^2_d)) \to \overline{\mathbb{B}_d}, \quad \chi \mapsto (\chi(z_1),\ldots,\chi(z_d)),
  \end{equation*}
  with the property that $\pi^{-1}(w) = \{ \delta_w \}$ for each $w \in \mathbb{B}_d$;
  see for instance Lemma 8.1 and Proposition 8.6 in \cite{Hartz17a}.

  Let
  \begin{equation*}
    b: \mathbb{D} \to \overline{\mathbb{B}_d}, \quad \lambda \mapsto (\pi \circ \Psi^*)(\delta_\lambda)
    = ( \Psi(z_1)(\lambda),\ldots, \Psi(z_d)(\lambda)).
  \end{equation*}
  Since $\Psi$ takes values in $\Mult(\mathcal{D})$, the map $b$ is holomorphic.
  If $b$ were constant, then since $\Psi(z_k) = \Phi(z_k \big|_V)$, injectivity
  of $\Phi$ would imply that $V$ is a singleton, a contradiction.
  Hence, $b$ is not constant and therefore takes values in the open ball $\mathbb{B}_d$
  by the maximum modulus principle.
  Since the fibers of $\pi$ over $\mathbb{B}_d$ are singletons, we see that
  $\Psi^*(\delta_\lambda) = \delta_{b(\lambda)}$ for all $\lambda \in \mathbb{D}$.
  Thus, for $\varphi \in \Mult(H^2_d)$ and $\lambda \in \mathbb{D}$, we find that
  \begin{equation*}
    \Psi(\varphi)(\lambda) = \Psi^*(\delta_\lambda)(\varphi) = \delta_{b(\lambda)}(\varphi) =
    \varphi(b(\lambda)),
  \end{equation*}
  so
  \begin{equation*}
    \Psi: \Mult(H^2_d) \to \Mult(\mathcal{D}), \quad \varphi \mapsto \varphi \circ b,
  \end{equation*}
  is a surjective homomorphism. But according to Theorem \ref{T:main}, this is impossible.
  This contradiction finishes the proof.
\end{proof}

\section{An explicit embedding for the Dirichlet space}
\label{sec:explicit_embedding}

Recall from the introduction that
\begin{equation*}
  k(z,w) = \frac{1}{z \overline{w}} \log \Big( \frac{1}{1 - z\overline{w}} \Big) =
  \sum_{n=0}^\infty (n+1)^n (z \overline{w})^n
\end{equation*}
denotes the reproducing kernel of the Dirichlet space and that the real numbers $(c_n)_{n=1}^\infty$
are defined by the power series identity
\begin{equation*}
  \sum_{n=1}^\infty c_n (z \overline{w})^n = 1 - \frac{1}{k(z,w)}.
\end{equation*}

An argument from the work of Kluyver \cite{Kluyver24} yields an explicit formula
for $c_n$. This is Proposition \ref{prop:explicit_intro}, which we restate for the reader's convenience.

\begin{proposition}
  \label{prop:explicit}
  For $n \ge 1$, we have
  \begin{equation*}
    c_n = \int_0^1 t \frac{\Gamma(n-t)}{\Gamma(n+1) \Gamma(1-t)} \, dt \ge 0.
  \end{equation*}
\end{proposition}

\begin{proof}
  We reproduce the computation from \cite{Kluyver24}.
Let $z \in \mathbb{D}$ and notice that
\begin{equation*}
  \frac{d}{d t} \frac{(1-z)^t}{\log(1-z)} = (1-z)^t,
\end{equation*}
hence
\begin{equation*}
    \sum_{n=1}^\infty c_n z^n =
    1 + \frac{z}{\log(1-z)} = 1 - \int_0^1 (1-z)^t \, d t.
\end{equation*}
We can expand the integrand into a binomial series 
\begin{equation*}
  (1-z)^t = \sum_{n=0}^\infty \binom{t}{n} (-1)^n z^n,
\end{equation*}
which converges uniformly in $t \in [0,1]$ for fixed $z \in \mathbb{D}$
because the binomial coefficients are bounded in modulus by $1$. Therefore,
\begin{equation*}
  \sum_{n=1}^\infty c_n z^n =
    \sum_{n=1}^\infty z^n (-1)^{n+1} \int_0^1 \binom{t}{n} \, dt,
\end{equation*}
so comparing coefficients, we conclude that for all $n \ge 1$,
  \begin{align*}
    c_n = (-1)^{n+1} \int_0^1 \binom{t}{n} \, dt &= \frac{1}{n!}
    \int_0^1 t (1-t) (2 - t) \ldots (n-1 - t) \, dt \\
                                                 &=
    \int_0^1 t \frac{\Gamma(n-t)}{\Gamma(n+1) \Gamma(1-t)} \, dt.
  \end{align*}
  Clearly, the integrand is non-negative, so $c_n \ge 0$ for all $n \ge 1$.
\end{proof}

The coefficients $G_n = (-1)^{n-1} c_n$ appear in the literature under the name Gregory coefficients;
see \cite{Blagouchine16} for historical remarks and many results regarding these coefficients.
In particular, the asymptotic behavior of the Gregory coefficients, and hence of $(c_n)$, is well understood; see Equation 52 in \cite{Blagouchine16}.
In the study of the Dirichlet space, knowledge about the asymptotic behavior of $(c_n)$ is sometimes useful;
see for instance \cite[p.\ 126]{Trent04}.
Here, we sketch how the formula in Proposition \ref{prop:explicit} can be used to determine the first-order
behavior of $(c_n)$, which also follows from the finer analysis in \cite{Blagouchine16}.

\begin{corollary}
  The asymptotic relation
  \begin{equation*}
    c_n \sim \frac{1}{n \log(n)^2}
  \end{equation*}
  holds.
\end{corollary}

\begin{proof}
  To determine the asymptotic behavior of the integrand in Proposition \ref{prop:explicit} as $n \to \infty$,
  we use the following inequality of Wendel \cite{Wendel48}:
  \begin{equation}
    \label{eqn:Wendel}
    \Big( \frac{x}{x+s} \Big)^{1-s} \le \frac{\Gamma(x+s)}{x^s \Gamma(x)} \le 1 \quad \text{ for all } x \in \mathbb{R}, s \in [0,1]
  \end{equation}
  In combination with the functional equation $\Gamma(x+1) = x \Gamma(x)$, this formula shows that
\begin{equation*}
  \lim_{n \to \infty} \frac{\Gamma(n-t)}{\Gamma(n+1)} (n+1)^{t+1} = 1
\end{equation*}
uniformly in $t \in [0,1]$, so by Proposition \ref{prop:explicit}, we find that
\begin{equation}
  \label{eqn:c_n_2}
  c_n \sim \int_{0}^1 t (n+1)^{-t-1} \frac{1}{\Gamma(1-t)} \, dt.
\end{equation}
The second inequality in \eqref{eqn:Wendel}, applied with $x=1$, shows that $s \Gamma(s) \le 1$ for $s \in [0,1]$.
Together with the basic inequality $\Gamma(s) \ge 1$ for $s \in [0,1]$, we find that
\begin{equation*}
  \int_{0}^1 t (1-t) (n+1)^{-t-1} \, dt \le
  \int_{0}^1 t (n+1)^{-t-1} \frac{1}{\Gamma(1-t)} \, dt \le
  \int_{0}^1 t (n+1)^{-t-1} \, dt.
\end{equation*}
Both of these integrals can be computed explicitly, for instance using integration by parts,
which shows that both integrals have asymptotic behavior $\sim n^{-1} \log(n)^{-2}$.
By \eqref{eqn:c_n_2}, the sequence $(c_n)$ therefore has the same asymptotic behavior.
\end{proof}

\bibliographystyle{amsplain}
\bibliography{../../../Dropbox/Literature/jabref_database}

\end{document}